\documentclass[a4paper, 12pt, reqno]{amsart}

\usepackage{fullpage}
\usepackage{amsmath, amsthm, amssymb, mathtools, mathrsfs}
\usepackage{hyperref}
\usepackage[capitalize,nameinlink,noabbrev,nosort]{cleveref}

\usepackage{float}
\usepackage{lscape}
\usepackage{adjustbox}
\usepackage{rotating}
\usepackage{bm}
\usepackage{chessfss}
\usepackage{mathtools}
\usepackage{tikz}
\usetikzlibrary{intersections, calc, arrows.meta}

\usepackage{pgfplots}

\usepackage{graphicx}
\usepackage{here}
\usepackage{time}

\usepackage{pstricks-add}
\usepackage{xcolor}

\hypersetup{
	colorlinks=true,       
	linkcolor=brown,          
	citecolor=brown,        
	filecolor=brown,      
	urlcolor=brown,           
}

\makeatletter
\@namedef{subjclassname@2020}{%
  \textup{2020} Mathematics Subject Classification}
\makeatother

\newtheorem{theoremcounter}{Theorem Counter}[section]

\theoremstyle{definition}
\newtheorem{definition}[theoremcounter]{Definition}
\newtheorem{remark}[theoremcounter]{Remark}
\newtheorem{example}[theoremcounter]{Example}

\theoremstyle{plain}
\newtheorem{lemma}[theoremcounter]{Lemma}

\newtheorem{theorem}[theoremcounter]{Theorem}

\numberwithin{equation}{section}

\begin{document}

\title[]{An alternative proof of the asymptotic formula for the Fourier coefficients of the elliptic modular $j$-function} 

\author{Karin Ikeda} 
\address{Joint Graduate School of Mathematics for Innovation, Kyushu University,
Motooka 744, Nishi-ku, Fukuoka 819-0395, Japan}
\email{ikeda.karin.236@s.kyushu-u.ac.jp}



\maketitle

\begin{abstract}
	In 1997, B\'aez-Duarte gave a probabilistic proof of the asymptotic formula for the partition function, which had originally been proved by Hardy--Ramanujan. Based on the probabilistic approach, this paper proves an asymptotic formula for the coefficients of the  elliptic modular $j$-function using various expressions in terms of modular functions having simple infinite products.
 \end{abstract}

\section{Introduction}
Let 
\begin{align*}
j(\tau)
:=\frac{E_4(\tau)^3}{\eta(\tau)^{24}}=\frac{1}{q}+744+196884 q+21493760q^2+\cdots\quad (q:=e^{2\pi i\tau}),
\end{align*}
be the elliptic modular $j$-function, where 
$$
E_4(\tau):=1+240\sum_{n=1}^{\infty}\sigma_3(n)q^n
$$
is the standard Eisenstein series of weight $4$, with $\sigma_3(n):=\sum_{d|n}d^3$,
$$
\eta(\tau):=q^{\frac{1}{24}}\prod_{n=1}^{\infty}(1-q^{n})
$$
is the Dedekind eta function, and $\tau\in\mathfrak{H}:=\{z\in\mathbb{C};\Im (z)>0\}.$ We denote by $c_n$ the $n$-th Fourier coefficient of $j(\tau)$:
$$
j(\tau)=\frac{1}{q}+744+\sum_{n=1}^{\infty}c_nq^n.
$$
Our goal of this paper is to give alternative proofs of the asymptotic formula for the Fourier coefficients of the $j$-function
\begin{align}\label{jasy}
c_n\sim\frac{e^{4\pi\sqrt{n}}}{\sqrt{2}n^{\frac{3}{4}}}\qquad  (n\to \infty).
\end{align}

There are various studies on the Fourier coefficients $c_n$ of the $j$-function, including  investigations into their relationship with the ``Monster" simple group~\cite{CN} and explicit formulas using traces of singular moduli~\cite{K}.
We especially focus on the asymptotic formula~\eqref{jasy} for the Fourier coefficients $c_n$ of the $j$-function. This formula was independently proved by Petersson~\cite{P} and Rademacher~\cite{R} using the circle method (actually they proved an infinite series expression for $c_n$ from which the asymptotic formula follows).
Later, the asymptotic formula has been proved in various ways. For example, R.~Murty--Sampath~\cite{MS} gave a proof based on the above-mentioned explicit formula for the coefficients $c_n$, while an alternative approach is given in \cite{DM}.

In this paper, we deduce the asymptotic formula \eqref{jasy} via probability theory based on B\'aez-Duarte's method~\cite{D}.

\section{Preliminaries for Probabilistic Methods}\label{sec2}
This section reviews the ideas of B\'aez-Duarte~\cite{D}.
Let 
$$
F(t)=\sum_{n=0}^{\infty}f_nt^n
$$
be power series with radius of convergence $R$, such that each of the coefficients $f_n$ is a real number, with $f_n\ge0$ for all $n$. We begin by reviewing B\'aez-Duarte's method as applied to the series $F(t)$. For each $t\in(0, R)$, let $X_t$ be a $\mathbb{Z}_{\ge 0}$-valued random variable such that
\begin{align}\label{sokudo}
P[X_t=n]:=\frac{f_nt^n}{F(t)}.
\end{align}
For any function $\psi : \mathbb{Z}_{\ge 0}\to\mathbb{C}$, we define the expected value as
$$
E[\psi(X_t)]:=\sum_{n=0}^{\infty}\psi(n)P[X_t=n].
$$
The mean $E[X_t]$ is denoted by $m_{F}(t)$ (or simply $m(t)$), and the variance $E[(X_t-m(t))^2]$ by $\sigma^2_{F}(t)$ or $\sigma^2(t)$. Then, it is then easily seen that
\begin{align}\label{kitaichi}
    m(t)=t\frac{d}{dt}\log F(t),\qquad \sigma^2(t)=t\frac{d}{dt}m(t).
\end{align}
Moreover, the characteristic function of $X_t$, defined as $E[e^{i\theta X_t}]$, for $\theta\in\mathbb{R}$, is given by
$$
E[e^{i\theta X_t}]=\frac{F(e^{i\theta}t)}{F(t)}.
$$

From now on, we will set
$$
X(t):= X_t\quad\text{and}\quad Z(t):=\frac{X(t)-m(t)}{\sigma(t)}.
$$

\begin{definition}
We say that $F(t)$ satisfies the \emph{strong Gaussian condition} if
$$
\lim_{ t\to R}\int_{-\pi\sigma(t)}^{\pi\sigma(t)}|E[e^{i\theta Z(t)}]-e^{-\frac{1}{2}\theta^2}|d\theta = 0
$$
holds.
\end{definition}
The following lemma gives the asymptotic formula for the coefficients of $F(t)$.

\begin{lemma} \label{1}
Let $F(t)$ be a power series satisfying the strong Gaussian condition. Suppose $\widetilde{m}(t)$ and $\widetilde{\sigma}(t)$ are continuous functions such that  $\widetilde{m}(t)\to \infty\ (t\to R)$ monotonically, $\widetilde{\sigma}(t)/\sigma(t)\to 1$ as  $t\to R$, and
\begin{align*}
\epsilon(t):=\frac{\widetilde{m}(t)-m(t)}{\sigma(t)}\to 0\quad(t\to R).
\end{align*}
Then, as $n\to \infty$,
$$
f_n\sim\frac{F(\tau_n)}{\sqrt{2\pi}\widetilde{\sigma}(\tau_n)\tau_n^n},
$$
where $\tau_n$ is the unique solution of the equation $\widetilde{m}(\tau_n)=n$ for large enough $n$.
\end{lemma}

\begin{proof}
We put
$$
\widetilde{Z}(t):=\frac{X(t)-\widetilde{m}(t)}{\widetilde{\sigma}(t)}.
$$
Since 
$$
F(\tau_ne^{i\theta})=\sum_{k=0}^{\infty}(f_k\tau_n^k)e^{ik\theta},
$$
we obtain 
$$
f_n=\frac{1}{2\pi\tau_n^n}\int_{-\pi}^{\pi}F(\tau_ne^{i\theta})e^{-in\theta}d\theta.
$$
Noting $E[e^{i\theta X_t}]=F(e^{i\theta}t)/F(t)$ and using the definition of $\widetilde{Z}(t)$ and $\widetilde{m}(\tau_n)=n$ together with the linearity of expectation, we obtain the following formula.
\begin{align}\label{keylem1}
\begin{split}
f_n
&=\frac{1}{2\pi \tau_n^n}\int_{-\pi}^{\pi}F(\tau_n)E[e^{i\theta X(\tau_n)}]e^{-in\theta}d\theta\\
&=\frac{F(\tau_n)}{2\pi\tau_n^n}\int_{-\pi}^{\pi}E[e^{i\theta(\widetilde{\sigma}(\tau_n)\widetilde{Z}(\tau_n)+\widetilde{m}(\tau_n))}]e^{-in\theta}d\theta\\
&=\frac{F(\tau_n)}{2\pi\widetilde{\sigma}(\tau_n)\tau_n^n}\int_{-\pi\widetilde{\sigma}(\tau_n)}^{\pi\widetilde{\sigma}(\tau_n)}E[e^{i\theta \widetilde{Z}(\tau_n)}]d\theta.
\end{split}
\end{align}

Furthermore, as for the last integral, 
\begin{align}\label{keylem2}
\begin{split}
\int_{-\pi\widetilde{\sigma}(\tau_n)}^{\pi\widetilde{\sigma}(\tau_n)}E[e^{i\theta \widetilde{Z}(\tau_n)}]d\theta
&=\frac{\widetilde{\sigma}(\tau_n)}{\sigma(\tau_n)}\int_{-\pi\sigma(\tau_n)}^{\pi\sigma(\tau_n)}E\left[e^{i\frac{\widetilde{\sigma}(\tau_n)}{\sigma(\tau_n)}\theta \widetilde{Z}(\tau_n)}\right]d\theta\\
&=\frac{\widetilde{\sigma}(\tau_n)}{\sigma(\tau_n)}\int_{-\pi\sigma(\tau_n)}^{\pi\sigma(\tau_n)}E\left[e^{i\frac{\widetilde{\sigma}(\tau_n)}{\sigma(\tau_n)}\theta \frac{X(\tau_n)-\widetilde{m}(\tau_n)}{\widetilde{\sigma}(\tau_n)}}\right]d\theta\\
&=\frac{\widetilde{\sigma}(\tau_n)}{\sigma(\tau_n)}\int_{-\pi\sigma(\tau_n)}^{\pi\sigma(\tau_n)}E\left[e^{\frac{i}{\sigma(\tau_n)}(\sigma(\tau_n)Z(\tau_n)+m(\tau_n))\theta}e^{-i\theta\frac{\widetilde{m}(\tau_n)}{\sigma(\tau_n)}}\right]d\theta\\
&=\frac{\widetilde{\sigma}(\tau_n)}{\sigma(\tau_n)}\int_{-\pi\sigma(\tau_n)}^{\pi\sigma(\tau_n)}E[e^{i\theta Z(\tau_n)}]e^{i\theta\frac{m(\tau_n)-\widetilde{m}(\tau_n)}{\sigma(\tau_n)}}d\theta\\
&=\frac{\widetilde{\sigma}(\tau_n)}{\sigma(\tau_n)}\int_{-\pi\sigma(\tau_n)}^{\pi\sigma(\tau_n)}E[e^{i\theta Z(\tau_n)}]e^{-i\theta\epsilon(\tau_n)}d\theta
\end{split}
\end{align}
holds. Here, noting that the strong Gaussian condition and 
$\epsilon(t)\to 0\ (t\to R)$ imply 
$$
\int_{-\pi\sigma(\tau_n)}^{\pi\sigma(\tau_n)}E[e^{i\theta Z(\tau_n)}]e^{-i\epsilon(\tau_n)\theta}d\theta\to\int_{-\infty}^{\infty}e^{-\frac{1}{2}\theta^2}d\theta\ (=\sqrt{2\pi})\quad(\tau_n\to R),
$$
and since
$$
f_n=\frac{F(\tau_n)}{2\pi\widetilde{\sigma}(\tau_n)\tau_n^n}\frac{\widetilde{\sigma}(\tau_n)}{\sigma(\tau_n)}\int_{-\pi\sigma(\tau_n)}^{\pi\sigma(\tau_n)}E[e^{i\theta Z(\tau_n)}]e^{-i\epsilon(\tau_n)\theta}d\theta
$$
by equations \eqref{keylem1} and \eqref{keylem2} and $\widetilde{\sigma}(\tau_n)/\sigma(\tau_n)\to 1\ (n\to\infty)$, we obtain 
$$
f_n\sim\frac{F(\tau_n)}{\sqrt{2\pi}\widetilde{\sigma}(\tau_n)\tau_n^n}.
$$
\end{proof}

We can apply a similar argument to an alternating series 
$$
G(t):=\sum_{n=0}^{\infty}g_nt^n
$$
with $(-1)^ng_n\ge 0$ by defining the random variable $\widetilde{X_t}$ with probability distribution
$$
P[\widetilde{X_t}=n]:=\frac{(-1)^ng_nt^n}{G^{*}(t)},
$$
where we set 
$$
G^{*}(t):=\sum_{n=0}^{\infty}|g_n|t^n.
$$
Equation \eqref{kitaichi} holds in exactly the same way, and the following Lemma~\ref{coefasy2} corresponds to Lemma~\ref{1}.

\begin{lemma} \label{coefasy2}
Let $G(t)$ be an alternating power series and assume the series $G^{*}(t)$ defined above satisfies the strong Gaussian condition. Let $m(t)=m_{G^{*}}(t)$ and $\sigma^2(t)=\sigma^2_{G^{*}}(t)$ as in the beginning of this section. Suppose $\widetilde{m}(t)$ and $\widetilde{\sigma}(t)$ are continuous functions such that  $\widetilde{m}(t)\to \infty\ (t\to R)$ monotonically, $\widetilde{\sigma}(t)/\sigma(t)\to 1$ as  $t\to R$, and
\begin{align*}
\epsilon(t):=\frac{\widetilde{m}(t)-m(t)}{\sigma(t)}\to 0\quad(t\to R).
\end{align*}
Then, as $n\to \infty$,
$$
g_n\sim\frac{G(\tau_n)}{\sqrt{2\pi}\widetilde{\sigma}(\tau_n)(-\tau_n)^n},
$$
where $\tau_n$ is the unique solution of the equation $\widetilde{m}(\tau_n)=n$ for large enough $n$.
\end{lemma}

When using Lemma~\ref{1} or Lemma~\ref{coefasy2}, we must make sure that $F(t)$ or $G(t)$ satisfies the strong Gaussian condition. In order to check this, we apply the following Lyapunov central limit theorem.

\begin{theorem}{\cite[Theorems~26.3 and 27.3]{B}}\label{LCLT}
    Suppose that for each $n$, a sequence of independent 
    random variables $X_{n, 1},\ldots, X_{n, r_n}$ such that $r_n\to\infty$ as $n\to\infty$ is given. Let 
    \begin{align*}
        \sigma^2_{n, k}:=E[X^2_{n, k}],\quad\text{and}\quad s^2_n:=\sum_{k=1}^{r_n}\sigma^2_{n, k},
    \end{align*}
    and assume $E[X_{n, k}]=0$ for each $k=1,\ldots,r_n$.
    If 
    \begin{align}\label{Lcondi}
        \lim_{n\to\infty}\sum_{k=1}^{r_n}\frac{1}{s^{2+\delta}_n}E\left[|X_{n, k}|^{2+\delta}\right]=0
    \end{align}
    holds for some positive $\delta$, then 
    \begin{align*}
        E\left[e^{i\theta \frac{S_n}{s_n}}\right]\to e^{-\frac{\theta^2}{2}}\quad(n\to\infty),
    \end{align*}
    where
    $$
    S_n=X_{n, 1}+\cdots+X_{n, r_n}.
    $$
\end{theorem}

\section{Proof of the asymptotic formula~\eqref{jasy}}\label{se3}
First, we express the $j$-function by using the theta function.
\begin{lemma}[Kaneko] \label{theta}
For $\tau\in\mathfrak{H}$, the following holds:
$$
j(\tau)=2^7(\theta_0(\tau)^8+\theta_2(\tau)^8+\theta_3(\tau)^8)(\theta_0(\tau)^{-8}+\theta_2(\tau)^{-8}+\theta_3(\tau)^{-8}),
$$
where
\begin{align*}
\theta_0(\tau)&:=\sum_{n\in\mathbb{Z}}(-1)^nq^{\frac{n^2}{2}}=\prod_{n=1}^{\infty}(1-q^{n})(1-q^{n-\frac{1}{2}})^2,\\
\theta_2(\tau)&:=\sum_{n\in\mathbb{Z}}q^{\frac{1}{2}\left(n+\frac{1}{2}\right)^2}=2q^{\frac{1}{8}}\prod_{n=1}^{\infty}(1-q^{n})(1+q^n)^2,
\end{align*}
and
\begin{align*}
\theta_3(\tau):=\sum_{n\in\mathbb{Z}}q^{\frac{n^2}{2}}=\prod_{n=1}^{\infty}(1-q^n)(1+q^{n-\frac{1}{2}})^2
\end{align*}
are Jacobi's theta function.
\end{lemma}

\begin{proof}
It is known that (see e.g.~\cite{Z}, p.29)
$$
E_4(\tau)=\frac{1}{2}(\theta_0(\tau)^8+\theta_2(\tau)^8+\theta_3(\tau)^8).
$$
Therefore, since 
\begin{align}\label{theta1}
\eta(\tau)^3=\frac{1}{2}\theta_0(\tau)\theta_2(\tau)\theta_3(\tau)
\end{align}
as seen from the infinite products,
we have 
$$
j(\tau)=2^5\frac{(\theta_0(\tau)^8+\theta_2(\tau)^8+\theta_3(\tau)^8)^3}{(\theta_0(\tau)\theta_2(\tau)\theta_3(\tau))^8}.
$$
Substituting $X=\theta_0(\tau)^4$, $Y=\theta_2(\tau)^4$, and $Z=\theta_3(\tau)^4$ into the identity 
$$
X^4+Y^4+Z^4-2(X^2Y^2+Y^2Z^2+Z^2X^2)=(X+Y+Z)(X+Y-Z)(X-Y+Z)(X-Y-Z)
$$ and noting the well-known identity (see e.g.~\cite{Z}, (31))
\begin{align}\label{theta2}
\theta_3(\tau)^4=\theta_0(\tau)^4+\theta_2(\tau)^4,
\end{align}
we have
\begin{align*}
&(\theta_0(\tau)^8+\theta_2(\tau)^8+\theta_3(\tau)^8)^2\\
&=\theta_0(\tau)^{16}+\theta_2(\tau)^{16}+\theta_3(\tau)^{16}+2(\theta_0(\tau)^8\theta_2(\tau)^8+\theta_0(\tau)^8\theta_3(\tau)^8+\theta_2(\tau)^8\theta_3(\tau)^8)\\
&=4(\theta_0(\tau)^8\theta_2(\tau)^8+\theta_0(\tau)^8\theta_3(\tau)^8+\theta_2(\tau)^8\theta_3(\tau)^8).
\end{align*}
Therefore we obtain
\begin{align*}
j(\tau)
&=2^5\frac{4(\theta_0(\tau)^8+\theta_2(\tau)^8+\theta_3(\tau)^8)(\theta_0(\tau)^8\theta_2(\tau)^8+\theta_2(\tau)^8\theta_3(\tau)^8+\theta_3(\tau)^8\theta_0(\tau)^8)}{(\theta_0(\tau)\theta_2(\tau)\theta_3(\tau))^8}\\
&=2^7(\theta_0(\tau)^8+\theta_2(\tau)^8+\theta_3(\tau)^8)(\theta_0(\tau)^{-8}+\theta_2(\tau)^{-8}+\theta_3(\tau)^{-8}).
\end{align*}
\end{proof}

By expanding the right-hand side of the formula in Lemma~\ref{theta}, we can represent the $j$-function as follows.
\begin{align}
j(\tau)=2^7\cdot3+H_1(\tau)+H_2(\tau)+H_3(\tau),
\end{align}
where
\begin{align*}
H_1(\tau)&:=2^7\left(\left(\frac{\theta_0(\tau)}{\theta_2(\tau)}\right)^8+\left(\frac{\theta_3(\tau)}{\theta_2(\tau)}\right)^8\right)=\sum_{n=-1}^{\infty}h_{1, n}q^n\\
&=\frac{1}{q}+104+276q-2048q^2+11202q^3-49152 q^4+184024q^5-\cdots,\\
H_2(\tau)&:=2^7\left(\left(\frac{\theta_0(\tau)}{\theta_3(\tau)}\right)^8+\left(\frac{\theta_3(\tau)}{\theta_0(\tau)}\right)^8\right)=\sum_{n=0}^{\infty}h_{2, n}q^n\\
&=256 + 131072 q + 11534336 q^2+441974784 q^3+10208935936 q^4+\cdots, 
\end{align*}
and
\begin{align*}
H_3(\tau)
&:=2^7\left(\left(\frac{\theta_2(\tau)}{\theta_3(\tau)}\right)^8+\left(\frac{\theta_2(\tau)}{\theta_0(\tau)}\right)^8\right)=\sum_{n=1}^{\infty}h_{3, n}q^n\\
&=65536 q + 9961472 q^2 + 422313984 q^3+10036969472 q^4+166007275520 q^5\cdots.
\end{align*}
Applying Lemma~\ref{1}, we  investigate the asymptotic behavior of each coefficient $h_{1, n}$, $h_{2, n}$, and $h_{3, n}$, and obtain the following theorem by identifying the main term.
\begin{theorem}\label{main}
The Fourier coefficients $c_n$ of the $j$-function admits the following asymptotic formula.
$$
c_n\sim\frac{e^{4\pi\sqrt{n}}}{\sqrt{2}n^{\frac{3}{4}}},
$$
as $n\to \infty$.
\end{theorem}
For positive integers $m$ and $a$ satisfying $1\le a\le m$, we define 
\begin{align*}
P_{m, a}(t):=\prod_{n=0}^{\infty}(1-t^{mn+a})^{-1}
\end{align*}
and also set
$$
Q(t):=\prod_{n=1}^{\infty}(1+t^{2n-1}),\quad R(t):=\prod_{n=1}^{\infty}(1+t^{2n}).
$$
Hereafter, the probability measure for the functions that appear will be defined as in \eqref{sokudo}.
In preparation for the proof, we also calculate $\log P_{m, a}(e^{-\lambda})$, $\log Q(e^{-\lambda})$, and $\log R(e^{-\lambda})$ for some $\lambda>0$.
Here and in the following, we denote $\rho:=-\log t$, and consider the estimates in the cases $\rho\to0$ and $\lambda\to0$.

Regarding the function $P_{m, a}(t)$, each quantity has already been described in~\cite{IM} as follows.
\begin{align}\label{mP}
m_{P_{m, a}}(t)&=\frac{\pi^2}{6m\rho^2}+\mathcal{O}\left(\frac{1}{\rho}\right),\\\label{oP}
\sigma_{P_{m, a}}^2(t)&=\frac{\pi^2}{3m\rho^3}+\mathcal{O}\left(\frac{1}{\rho^2}\right),\\\label{lP}
\log P_{m,a}(e^{-\lambda})&=\frac{\pi^2}{6m\lambda}+\log \frac{\Gamma\left(\frac{a}{m}\right)}{\sqrt{2\pi}}+\left(\frac{a}{m}-\frac{1}{2}\right)\log m\lambda+\mathcal{O}(\lambda).
\end{align}
For the functions $Q(t)$ and $R(t)$, we likewise evaluate the mean and the variance using the Euler--Maclaurin summation formula.

We first compute the mean of $Q(t)$ as follows.
\begin{align}\nonumber
m_Q(t)
&=t\frac{d}{dt}\log Q(t)=t\frac{d}{dt}\sum_{n=1}^{\infty}\log (1+t^{2n-1})=\sum_{n=1}^{\infty}\frac{(2n-1)e^{-(2n-1)\rho}}{1+e^{-(2n-1)\rho}}\\\nonumber
&=\frac{e^{-\rho}}{2(1+e^{-\rho})}+\int_{1}^{\infty}\frac{(2x-1)e^{-(2x-1)\rho}}{1+e^{-(2x-1)\rho}}dx\\\nonumber
&\qquad+\int_{1}^{\infty}\left(x-\lfloor x\rfloor-\frac{1}{2}\right)\frac{d}{dx}\left(\frac{(2x-1)e^{-(2x-1)\rho}}{1+e^{-(2x-1)\rho}}\right)dx\\\nonumber
&=\int_{1}^{\infty}\frac{(2x-1)e^{-(2x-1)\rho}}{1+e^{-(2x-1)\rho}}dx\\\label{mqeq}
&\qquad+\int_{1}^{\infty}\left(x-\lfloor x\rfloor-\frac{1}{2}\right)\frac{d}{dx}\left(\frac{(2x-1)e^{-(2x-1)\rho}}{1+e^{-(2x-1)\rho}}\right)dx+\mathcal{O}(1).
\end{align}
For the first term, we have
\begin{align*}
\int_{1}^{\infty}\frac{2x-1}{e^{(2x-1)\rho}+1}dx
&=\frac{1}{2\rho^2}\int_{\rho}^{\infty}\frac{x}{e^x+1}dx=\frac{1}{2\rho^2}\left(\int_{0}^{\infty}-\int_{0}^{\rho}\right)\frac{x}{e^x+1}dx\\
&=\frac{1}{8\rho^2}\int_{0}^{\infty}\left(\frac{xe^{\frac{x}{2}}}{e^x-1}-\frac{x}{e^x-1}\right)dx-\frac{1}{2\rho^2}\int_{0}^{\rho}\frac{x}{e^x+1}dx\\
&=\frac{1}{8\rho^2}\left(\zeta\left(2, \frac{1}{2}\right)-\zeta(2)\right)+\mathcal{O}(1)\\
&=\frac{\pi^2}{24\rho^2}+\mathcal{O}(1).
\end{align*}
Here, we use the integral representation
$$
\zeta(s, a)=\frac{1}{\Gamma(s)}\int_{0}^{\infty}\frac{t^{s-1}e^{(1-a)t}}{e^t-1}dt
$$
of the Hurwitz zeta function 
$$
\zeta(s,a):=\sum_{n=0}^{\infty}\frac{1}{(n+a)^s}\qquad(\Re(s)>1, 0<a\le 1)
$$ 
and the values $\zeta(2)=\pi^2/6$ and $\zeta(2, 1/2)=(2^2-1)\zeta(2)=\pi^2/2$. Next, for the second term of \eqref{mqeq}, we have
\begin{align*}
&\left|\int_{1}^{\infty}\left(x-\lfloor x\rfloor-\frac{1}{2}\right)\frac{d}{dx}\left(\frac{2x-1}{e^{(2x-1)\rho}+1}\right)dx\right|\\
&\le \int_{1}^{\infty}\left|\frac{d}{dx}\left(\frac{2x-1}{e^{(2x-1)\rho}+1}\right)\right|dx\\
&=\int_{1}^{\infty}\left|\frac{2}{e^{(2x-1)\rho}+1}-\frac{2\rho(2x-1)e^{(2x-1)\rho}}{(e^{(2x-1)\rho}+1)^2}\right|dx\\
&\le\int_{1}^{\infty}\left(\left|\frac{2}{e^{(2x-1)\rho}+1}\right|+\left|\frac{2\rho(2x-1)e^{(2x-1)\rho}}{(e^{(2x-1)\rho}+1)^2}\right|\right)dx \\
&=\frac{\log (1+e^{-\rho})}{\rho}+\frac{1}{2\rho}\left(\int_{0}^{\infty}-\int_{0}^{\rho}\right)\frac{2xe^{x}}{(e^{x}+1)^2}dx\\
&\le  \frac{\log (1+e^{-\rho})}{\rho}+\frac{\log 2}{\rho}+\frac{1}{2\rho}\rho\frac{2\rho e^{\rho}}{(e^{\rho}+1)^2}\le \frac{e^{-\rho}}{\rho}+\frac{\log 2}{\rho}+\frac{e^{\rho}}{\rho}.
\end{align*}
Here we have used 
\begin{align*}
\int_{0}^{\infty}\frac{xe^{x}}{(e^{x}+1)^2}dx=\left[\frac{x}{e^x+1}\right]_{0}^{\infty}+\int_{0}^{\infty}\frac{dx}{e^x+1}=\log2
\end{align*}
and the fact that the function $xe^{x}/(e^x+1)^2$ is increasing in $[0, \rho]$ for small $\rho$.
Hence, we obtain the following estimate for $m_Q(t)$.
\begin{align}\label{mQ}
m_Q(t)=\frac{\pi^2}{24\rho^2}+\mathcal{O}\left(\frac{1}{\rho}\right)\qquad(\rho\to 0).
\end{align}
Applying the same method, we obtain the following estimate for the variance of $Q(t)$.
\begin{align}\nonumber
\sigma^2_Q(t)
&=t\frac{d}{dt}m_Q(t)=\sum_{n=1}^{\infty}\frac{(2n-1)^2e^{(2n-1)\rho}}{(e^{(2n-1)\rho}+1)^2}\\\nonumber
&=\frac{e^{\rho}}{2(e^{\rho}+1)^2}+\int_{1}^{\infty}\frac{(2x-1)^2e^{(2x-1)\rho}}{(e^{(2x-1)\rho}+1)^2}dx\\\nonumber
&\qquad+\int_{1}^{\infty}\left(x-\lfloor x\rfloor-\frac{1}{2}\right)\frac{d}{dx}\left(\frac{(2x-1)^2e^{(2x-1)\rho}}{(e^{(2x-1)\rho}+1)^2}\right)dx\\\nonumber
&=\int_{1}^{\infty}\frac{(2x-1)^2e^{(2x-1)\rho}}{(e^{(2x-1)\rho}+1)^2}dx\\\nonumber
&\qquad+\int_{1}^{\infty}\left(x-\lfloor x\rfloor-\frac{1}{2}\right)\frac{d}{dx}\left(\frac{(2x-1)^2e^{(2x-1)\rho}}{(e^{(2x-1)\rho}+1)^2}\right)dx+\mathcal{O}\left(\frac{1}{\rho^2}\right).\\\nonumber
\end{align}
For the first term, we have
\begin{align*}
\int_{1}^{\infty}\frac{(2x-1)^2e^{(2x-1)\rho}}{(e^{(2x-1)\rho}+1)^2}dx
&=\frac{1}{2\rho^3}\left(\int_{0}^{\infty}-\int_{0}^{\rho}\right)x^2\frac{d}{dx}\left(-\frac{1}{e^x+1}\right)dx\\
&=\frac{1}{\rho^3}\int_{0}^{\infty}\frac{x}{e^x+1}dx+\mathcal{O}\left(\frac{1}{\rho}\right)\\
&=\frac{\pi^2}{12\rho^3}+\mathcal{O}\left(\frac{1}{\rho}\right).
\end{align*}
For the second term, we obtain the following by applying the same method as in the evaluation of the second term of \eqref{mqeq}.
\begin{align*}
&\left|\int_{1}^{\infty}\left(x-\lfloor x\rfloor-\frac{1}{2}\right)\frac{d}{dx}\left(\frac{(2x-1)^2e^{-(2x-1)\rho}}{(1+e^{-(2x-1)\rho})^2}\right)dx\right|\\
&=\int_{1}^{\infty}\left|\frac{4(2x-1)e^{(2x-1)\rho}}{(e^{(2x-1)\rho}+1)^2}+\frac{2\rho(2x-1)^2e^{(2x-1)\rho}(1-e^{(2x-1)\rho})}{(e^{(2x-1)\rho}+1)^3}\right|dx\\
&\le \int_{1}^{\infty}\frac{4(2x-1)e^{(2x-1)\rho}}{(e^{(2x-1)\rho}+1)^2}+\frac{2\rho(2x-1)^2e^{(2x-1)\rho}}{(2x-1)\rho(e^{(2x-1)\rho}+1)^2}+\frac{2\rho(2x-1)^2e^{2(2x-1)\rho}}{(e^{(2x-1)\rho}+1)^3}dx\\
&=\int_{1}^{\infty}\frac{6(2x-1)e^{(2x-1)\rho}}{(e^{(2x-1)\rho}+1)^2}dx+\frac{1}{\rho^2}\left(\int_{0}^{\infty}-\int_{0}^{\rho}\right)\frac{x^2e^{2x}}{(e^x+1)^3}dx\\
&\le \frac{6}{\rho}\left(\frac{\log 2}{2\rho}+\frac{e^{\rho}}{2\rho}\right)+\frac{1}{\rho^2}\left(\frac{\pi^2}{12}+\log 2\right)+\frac{1}{\rho}\frac{\rho^2 e^{2\rho}}{(e^\rho+1)^3}\\
&\le \frac{6}{\rho}\left(\frac{\log 2}{2\rho}+\frac{e^{\rho}}{2\rho}\right)+\frac{1}{\rho^2}\left(\frac{\pi^2}{12}+\log 2\right)+\frac{e^{2\rho}}{\rho^2}.
\end{align*}
Therefore, we obtain the following.
\begin{align}\label{oQ}
\sigma_{Q}^2(t)=\frac{\pi^2}{12\rho^3}+\mathcal{O}\left(\frac{1}{\rho^2}\right)\qquad(\rho\to 0).
\end{align}
Furthermore, by applying the Euler--Maclaurin summation formula, we obtain 
\begin{align}\nonumber
\log Q(e^{-\lambda})
&=\frac{\log(1+e^{-\lambda})}{2}+\int_{1}^{\infty}\log (1+e^{-(2x-1)\lambda})dx\\\nonumber
&\qquad\qquad-\int_{1}^{\infty}\frac{2\lambda}{e^{(2x-1)\lambda}+1}\left(x-\lfloor x\rfloor-\frac{1}{2}\right)dx\\\nonumber
&=\log (1+e^{-\lambda})+\lambda\int_{0}^{\infty}\frac{2x-1}{e^{(2x-1)\lambda}+1}dx\\\nonumber
&\qquad-\int_{0}^{1}\log (1+e^{-(2x-1)\lambda})dx-\int_{1}^{\infty}\frac{2\lambda}{e^{(2x-1)\lambda}+1}\left(x-\lfloor x\rfloor-\frac{1}{2}\right)dx\\\nonumber
&=\frac{\pi^2}{24\lambda}-\frac{1}{2\lambda}\int_{-\lambda}^{\lambda}\frac{x}{e^{x}+1}dx-\int_{1}^{\infty}\frac{2\lambda}{e^{(2x-1)\lambda}+1}\left(x-\lfloor x\rfloor-\frac{1}{2}\right)dx\\\label{lQ}
&=\frac{\pi^2}{24\lambda}+\mathcal{O}(\lambda)+\mathcal{O}(\lambda e^{-\lambda})=\frac{\pi^2}{24\lambda}+\mathcal{O}(\lambda)\qquad(\lambda\to 0).
\end{align}
Similarly to the case of $Q(t)$, we obtain the following for $R(t)$ by applying the Euler--Maclaurin formula and the integral representation of the Riemann zeta function.
\begin{align}\label{mR}
m_{R}(t)&=\frac{\pi^2}{24\rho^2}+\mathcal{O}\left(\frac{1}{\rho}\right),\\\label{oR}
\sigma_{R}^2(t)&=\frac{\pi^2}{12\rho^3}+\mathcal{O}\left(\frac{1}{\rho^2}\right),\\\label{lR}
\log R(e^{-\lambda})&=\frac{\pi^2}{24\lambda}-\frac{1}{2}\log2+\mathcal{O}(\lambda).
\end{align}
Using these computational results, we analyze the asymptotic behavior of the coefficients in the $q$-expansions of the functions $H_1(\tau)$, $H_2(\tau)$, and $H_3(\tau)$.\\

\noindent\textbf{Evaluation of the coefficient of} \bm{$H_1$} \textbf{.}

First, we consider an alternative representation of the function $H_1(\tau)$. Combining \eqref{theta1} and \eqref{theta2} yields 
\begin{align*}
H_1(\tau)
&=2^7\frac{(\theta_0(\tau)^8+(\theta_0(\tau)^4+\theta_2(\tau)^4)^2)\theta_2(\tau)^4}{\theta_2(\tau)^{12}}\\&=2^7\frac{\theta_2(\tau)^{12}+2\theta_0(\tau)^4\theta_2(\tau)^4(\theta_0(\tau)^4+\theta_2(\tau)^4)}{\theta_2(\tau)^{12}}\\
&=2^7\left(1+2\frac{(2\eta(\tau)^3)^4}{\theta_2(\tau)^{12}}\right)=2^7+\frac{1}{q}\prod_{n=1}^{\infty}(1+q^n)^{-24}.
\end{align*}

Next, we define the order of power series described below.
\begin{definition}\label{zyunzyo}
For given two power series
$$
\sum_{n=0}^{\infty}s_nz^n\qquad\text{and}\qquad\sum_{n=0}^{\infty}t_nz^n,
$$
we denote by
$$
\sum_{n=0}^{\infty}s_nz^n\prec\sum_{n=0}^{\infty}t_nz^n
$$
if there exists an integer $N$ such that $s_n\le t_n$ for all $n\ge N$.
\end{definition}

Since
$$
\frac{1}{q}\prod_{n=1}^{\infty}(1+q^{n})^{-24}\prec\frac{1}{q}\prod_{n=1}^{\infty}(1-q^{n})^{-24},
$$
we aim to apply Lemma~\ref{1} to 
$$
H_1^{*}(t):=\prod_{n=1}^{\infty}(1-t^{n})^{-24}=\sum_{n=0}^{\infty}a^{*}_{n}t^n.
$$
This allows us to evaluate the asymptotic formula of the coefficient in the $q$-expansion of $H_1(\tau)$. Since it has been shown that $P_{m, a}(t)$ satisfies the strong Gaussian condition in \cite[Theorem~3.2]{IM}, the function $H_1^{*}(t)$ also satisfies the strong Gaussian condition.
From \eqref{mP} and \eqref{oP}, the mean $m_{H_1^{*}}(t)$ and the variance $\sigma^2_{H_1^{*}}(t)$ of $H_1^{*}(t)$ are as follows:
\begin{align*}
m_{H_1^{*}}(t)
&=t\frac{d}{dt}\log H_1^{*}(t)=24m_{P_{1, 1}}(t)=24\left(\frac{\pi^2}{6\rho^2}+\mathcal{O}\left(\frac{1}{\rho}\right)\right)\\
&=\frac{4\pi^2}{\rho^2}+\mathcal{O}\left(\frac{1}{\rho}\right),\\
\sigma^2_{H_1^{*}}(t)
&=t\frac{d}{dt}m_{H_1^{*}}(t)=24\sigma^2_{P_{1, 1}}(t)=\frac{8\pi^2}{\rho^3}+\mathcal{O}\left(\frac{1}{\rho^2}\right).
\end{align*}

Let 
\begin{align*}
\widetilde{m}_{H_1^{*}}(t):=\frac{4\pi^2}{\rho^2}\qquad\text{and}\qquad\widetilde{\sigma}_{H_1^{*}}(t):=\sqrt{\frac{8\pi^2}{\rho^3}}.
\end{align*}
These functions satisfy conditions $\widetilde{m}_{H_1^{*}}(t)\to\infty$, $\widetilde{\sigma}_{H_1^{*}}(t)/\sigma_{H_1^{*}}(t)\to1$ as $t\to 1$, and
\begin{align*}
\epsilon(t):=\frac{\widetilde{m}_{H_1^{*}}(t)-m_{H_1^{*}}(t)}{\sigma_{H_1^{*}}(t)}\to 0\quad(t\to 1)
\end{align*}
in Lemma~\ref{1}. From equation \eqref{lP}, we obtain
\begin{align*}
\log H_1^{*}(e^{-\lambda})
&=\frac{4\pi^2}{\lambda}+12\log\frac{\lambda}{2\pi}+\mathcal{O}(\lambda),
\end{align*}
and from Lemma~\ref{1}, we conclude
\begin{align}\label{A'}
h^{*}_{1, n}\sim\frac{e^{4\pi\sqrt{n}}}{\sqrt{2}n^{\frac{27}{4}}}.
\end{align}

\noindent\textbf{Evaluation of the coefficient of} \bm{$H_2$} \textbf{.}

The function $H_2$ can be reweitten as
\begin{align}\label{B'}
H_2(\tau)&=2^7\prod_{n=1}^{\infty}\left(\frac{1+(-q^{1/2})^{2n-1}}{1-(-q^{1/2})^{2n-1}}\right)^{16}+2^7\prod_{n=1}^{\infty}\left(\frac{1+(q^{1/2})^{2n-1}}{1-(q^{1/2})^{2n-1}}\right)^{16}\\\nonumber
&=(128 - 4096q^{\frac{1}{2}} + 65536 q - 704512 q^{\frac{3}{2}} + 5767168 q^2\cdots)\\\nonumber
&\qquad+(128 + 4096q^{\frac{1}{2}} + 65536 q + 704512 q^{\frac{3}{2}} + 5767168 q^2\cdots),
\end{align}
where we examine the coefficients of the first and second terms when they are expressed as $q$-series. Since the first term is an alternating series in $q^{1/2}$ and the second is a series with positive coefficients, it suffices to apply Lemma~\ref{1} to 
$$
H_2^{*}(t):=2^7\prod_{n=1}^{\infty}\left(\frac{1+t^{2n-1}}{1-t^{2n-1}}\right)^{16}=\sum_{n=0}^{\infty}h^{*}_{2, n}t^n.
$$
From \eqref{mP}, \eqref{oP}, \eqref{mQ}, and \eqref{oQ}, the mean $m_{H_2^{*}}(t)$ and the variance $\sigma^2_{H_2^{*}}(t)$ of the function $H_2^{*}(t)$ are computed as follows:
\begin{align}\label{mB'}
m_{H_2^{*}}(t)
&=t\frac{d}{dt}\log H_2^{*}(t)=16(m_{Q}(t)+m_{P_{2, 1}}(t))=\frac{2\pi^2}{\rho^2}+\mathcal{O}\left(\frac{1}{\rho}\right),\\\label{oB'}
\sigma^2_{H_2^{*}}(t)
&=t\frac{d}{dt}m_{H_2^{*}}(t)=16(\sigma^2_{Q}(t)+\sigma^2_{P_{2, 1}}(t))=\frac{4\pi^2}{\rho^3}+\mathcal{O}\left(\frac{1}{\rho^2}\right).
\end{align}
By using \eqref{mB'} and \eqref{oB'}, we prove the following theorem.
\begin{theorem}\label{SG}
The function $H_2^{*}(t)$ satisfies the strong Gaussian condition.
\end{theorem}
\begin{proof}
Let
$$
H^{*}_{2, k}(t):=\left(\frac{1+t^{2k-1}}{1-t^{2k-1}}\right)^{16}=\sum_{n=0}^{\infty}h^{*}_{2, n, k}t^n,
$$
and let $X_{t, k}$ be a random variable such that 
$$
P[X_{t, k}=n]=\frac{h^{*}_{2, n, k}t^n}{H^{*}_{2, k}(t)}.
$$
We denote the mean of $H^{*}_{2, k}(t)$ by $m_{k}(t)$ and the variance by $\sigma^2_{2, k}(t)$, and define $Y_{t, k}:=X_{t, k}-m_k(t)$. Then, we have $E[Y_{t, k}]=0$. To complete the proof, we compute $E[Y^4_{t, k}]$ to confirm that the condition \eqref{Lcondi} holds. For $n\ge 1$, since 
$$
E[X^n_{t, k}]=\frac{1}{H^{*}_{2, k}(t)}\left(t\frac{d}{dt}\right)^nH^{*}_{2, k}(t)
$$
is satisfied, 
$$
E[X_{t, k}^n]\le C_1\frac{(2k-1)^nt^{2k-1}}{(1-t^{2k-1})^n}\qquad(\exists C_1>0)
$$
holds, and thus we obtain
\begin{align*}
E[Y^4_{t, k}]
&=E[X^4_{t, k}]+m^4_{k}(t)-4m_{k}(t)E[X^3_{t, k}]-4m^3_{k}(t)E[X_{t, k}]+6m^2_{k}(t)E[X^2_{t, k}]\\
&\le C_2\frac{(2k-1)^4e^{-(2k-1)\rho}}{(1-e^{-(2k-1)\rho})^4}\qquad(\exists C_2>0).
\end{align*}
By applying the Euler--Maclaurin summation formula again, we obtain 
$$
\sum_{k=1}^{\infty}\frac{(2k-1)^4e^{-(2k-1)\rho}}{(1-e^{-(2k-1)\rho})^4}=-\frac{2\zeta(2)+6\zeta(3)+4\zeta(4)}{\rho^5}+\mathcal{O}\left(\frac{1}{\rho^4}\right).
$$
Thus,
$$
\lim_{t\to 1}\frac{1}{(\sigma^2_{H^{*}_{2}}(t))^2}\sum_{k=1}^{\infty}E[Y_{t, k}^4]\le\lim_{t\to 1}\left(\frac{\rho^3}{4\pi^2}\right)^2\frac{2\zeta(2)+6\zeta(3)+4\zeta(4)}{\rho^5}=0
$$
holds, and applying Theorem~\ref{LCLT} allows us to conclude that $H^{*}_{2}(t)$ satisfies the strong Gaussian condition. 
\end{proof}

We now define 
\begin{align*}
\widetilde{m}_{H^{*}_{2}}(t):=\frac{2\pi^2}{\rho^2}\qquad\text{and}\qquad\widetilde{\sigma}_{H^{*}_{2}}(t):=\sqrt{\frac{4\pi^2}{\rho^3}}.
\end{align*}
Since
\begin{align*}
\log H^{*}_{2}(e^{-\lambda})
&=\log 2^7+16\log P_{2, 1}(e^{-\lambda})+16\log Q(e^{-\lambda})\\
&=\log \frac{1}{2}+\frac{2\pi^2}{\lambda}+\mathcal{O}(\lambda),
\end{align*}
holds from \eqref{lP} and \eqref{lQ}, we obtain 
$$
h^{*}_{2, n}\sim\frac{e^{2\pi\sqrt{2n}}}{2(2n)^{\frac{3}{4}}}
$$
by Lemma~\ref{1}. Thus, by noting equtation \eqref{B'}, we conclude
\begin{align}
h_{2, n}\sim(1+(-1)^{2n})\frac{e^{4\pi\sqrt{n}}}{2(4n)^{\frac{3}{4}}}=\frac{e^{4\pi\sqrt{n}}}{2\sqrt{2}n^{\frac{3}{4}}}.
\end{align}


\noindent\textbf{Evaluation of the coefficient of} \bm{$H_3$} \textbf{.}
The function $H_3(\tau)$, as in the case of $H_2(\tau)$, has a product representation given by
\begin{align}\label{C'}
H_3(\tau)
&=2^{15}q\prod_{n=1}^{\infty}\left(\frac{1+(-q^{1/2})^{2n}}{1-(-q^{1/2})^{2n-1}}\right)^{16}+2^{15}q\prod_{n=1}^{\infty}\left(\frac{1+(q^{1/2})^{2n}}{1-(q^{1/2})^{2n-1}}\right)^{16}\\\nonumber
&=(32768 q - 524288 q^{\frac{3}{2}} + 4980736 q^2 - 35651584 q^{\frac{5}{2}} + 
 211156992 q^3-\cdots)\\\nonumber
&\qquad+(32768 q + 524288 q^{\frac{3}{2}} + 4980736 q^2 + 35651584 q^{\frac{5}{2}} + 211156992 q^3+\cdots),
\end{align}
Hence, we apply the same method used for the function $H_2(\tau)$ to compute the function $H_3(\tau)$ as well. We define the function $H_3^{*}(t)$ as 
$$
H_3^{*}(t):=2^{15}t^2\prod_{n=1}^{\infty}\left(\frac{1+t^{2n}}{1-t^{2n-1}}\right)^{16}=\sum_{n=0}^{\infty}h^{*}_{3,n}t^n
$$
and obtain the asymptotic behavior of $h_{3, n}$ by analyzing the asymptotic behavior of its coefficients. From \eqref{mP}, \eqref{oP}, \eqref{mR}, and \eqref{oR}, the mean $m_{H^{*}_3}(t)$ and the variance $\sigma^2_{H^{*}_3}(t)$ of the function $H^{*}_{3}(t)$ are calculated as follows:
\begin{align}
m_{H^{*}_{3}}(t)&=2+16(m_{R}(t)+m_{P_{2, 1}}(t))=2+\frac{2\pi^2}{\rho^2}+\mathcal{O}\left(\frac{1}{\rho}\right),\\\label{H3o}
\sigma^2_{H^{*}_{3}}(t)&=16(\sigma^2_{R}(t)+\sigma^2_{P_{2, 1}}(t))=\frac{4\pi^2}{\rho^3}+\mathcal{O}\left(\frac{1}{\rho^2}\right).
\end{align}
Using the product expansion of $H^{*}_{3}(t)$ as in Theorem~\ref{SG}, we define a probability measure for each 
$$
H^{*}_{3, k}(t):=\left(\frac{1+t^{2k}}{1-t^{2k-1}}\right)^{16}.
$$
Noting
$$
\sum_{k=1}^{\infty}\frac{k^4e^{-2k\rho}}{(1-e^{-(2k-1)\rho})^4}\le e^{-\rho}\sum_{k=1}^{\infty}\frac{(2k-1)^4e^{-(2k-1)\rho}}{(1-e^{-(2k-1)\rho})^4}
$$
and equation \eqref{H3o}, we can then verify by direct calculation that $H^{*}_{3}(t)$ also satisfies the condition \eqref{Lcondi}, and thus satisfies the strong Gaussian condition.
Thus, we define 
\begin{align*}
\widetilde{m}_{H^{*}_{3}}(t):=\frac{2\pi^2}{\rho^2}\qquad\text{and}\qquad\widetilde{\sigma}_{H^{*}_{3}}(t):=\sqrt{\frac{4\pi^2}{\rho^3}}.
\end{align*}
Then, since 
\begin{align*}
\log H^{*}_{3}(e^{-\lambda})
&=\log 2^{15}+16\log R(e^{-\lambda})+16\log P_{2, 1}(e^{-\lambda}) +\mathcal{O}(\lambda)\\
&=\log \frac{1}{2}+\frac{2\pi^2}{\lambda}+\mathcal{O}(\lambda).
\end{align*}
from \eqref{lP} and \eqref{lR}, we obtain 
$$
h^{*}_{3, n}\sim\frac{e^{2\pi\sqrt{2n}}}{2(2n)^{\frac{3}{4}}}
$$
by Lemma~\ref{1}. Hence, by equation \eqref{C'}, we conclude that 
\begin{align}\label{h3}
h_{3, n}\sim(1+(-1)^{2n})\frac{e^{4\pi\sqrt{n}}}{2(4n)^{\frac{3}{4}}}=\frac{e^{4\pi\sqrt{n}}}{2\sqrt{2}n^{\frac{3}{4}}}.
\end{align}

Therefore, from \eqref{A'}, \eqref{B'}, and \eqref{C'}, we see that the main terms come from the functions $H_2(\tau)$ and $H_3(\tau)$, which yields
$$
c_n\sim\frac{e^{4\pi\sqrt{n}}}{2\sqrt{2}n^{\frac{3}{4}}}+\frac{e^{4\pi\sqrt{n}}}{2\sqrt{2}n^{\frac{3}{4}}}=\frac{e^{4\pi\sqrt{n}}}{\sqrt{2}n^{\frac{3}{4}}}
$$
and completes the proof of Theorem~\ref{main}.
\section{Another proof via Hauptmodul}

As shown in Lemma~\ref{theta}, the $j$-function can be expressed using theta functions. 
But there are many other expressions of the $j$-function as rational functions of another modular functions.
In this section, we explain that the asymptotic formula for the Fourier coefficients of the $j$-function can also be obtained via such different expressions by the same method as in \cref{se3}.

For a positive integer $N$, let 
$$
\Gamma_{0}(N)=\left\{
\begin{pmatrix}
a & b \\
c & d \\
\end{pmatrix}
\in\Gamma\ ; c\equiv0\ (\text{mod}\ N)
\right\}
$$
be the Hecke congruence subgroup of level $N$. When the genus of $\Gamma_0(N)$ is zero, we call a generator of the modular function field of $\Gamma_{0}(N)$ having a simple pole at $q=0$ with residue $1$ and holomorphic in $\mathfrak{H}$ a Hauptmodul, and denote it by $j_N$. Examples of Hauptmodul are described in \cite[Table~3]{CN}, and they are represented by eta-quotients. Furthermore,  the $j$-function can be expressed as a rational function in terms of each Hauptmodul, as detailed in \cite[Chapter~5--Section~2]{F}. Based on \cite[Table~3]{CN} and \cite[Chapter~5--Section~2]{F}, we lists in \cref{Table1} only those Hauptmoduln that can be expressed as simplest eta-quotients.

\renewcommand{\arraystretch}{1.3}
\begin{table}[H]
\begin{center}
            \begin{tabular}{|c|c|c|c|} \hline
               $N$ &        Hauptmodul $j_N$    &                 Algebraic equation            \\ \hline
               $2$ &        $j_2=\eta(\tau)^{24}/\eta(2\tau)^{24}$       &                     $j(\tau)=(j_2+256)^3/j_2^2$                     \\ \hline
               $3$  &        $j_3=\eta(\tau)^{12}/\eta(3\tau)^{12}$       &                     $j(\tau)=(j_3+27)(j_3+243)^3/j_3^3$                       \\ \hline
               $4$  &        $j_4=\eta(\tau)^{8}/\eta(4\tau)^{8}$       &               $j(\tau)=(j_4^2+256j_4+4096)^3/j^4_4(j_4+16)$              \\ \hline
               $5$  &      $j_5=\eta(\tau)^{6}/\eta(5\tau)^{6}$      & $j(\tau)=(j_5^2+250j_5+3125)^3/j_5^5$   \\ \hline
               $7$  & $j_7=\eta(\tau)^{4}/\eta(7\tau)^{4}$ &    $j(\tau)=(j_7^2+13j_7+49)(j_7^2+245j_7+2401)^3/j_7^7$\\ \hline
               $9$  &        $j_9=\eta(\tau)^{3}/\eta(9\tau)^{3}$          &   \begin{tabular}{c}
                                            $j(\tau)=(j_9+9)^3(j_9^3+243j_9^2\phantom{XXXXXXX}$\\
                                                $\phantom{XXXXX}+2187j_9+6561)^3/j_9^9(j_9^2+9j_9+27)$
                                         \end{tabular}                \\ \hline
               $13$  &        $j_{13}=\eta(\tau)^{2}/\eta(13\tau)^{2}$    &   \begin{tabular}{c}
                                            $j(\tau)=(j_{13}^2+5j_{13}+13)(j_{13}^4+247j_{13}^3\phantom{XXXXXXXX}$\\
                                                $\phantom{XXXXXX}+3380j_{13}^2+15379j_{13}+28561)^3/j_{13}^{13}$
                                         \end{tabular}                  \\ \hline
               $25$  &        $j_{25}=\eta(\tau)/\eta(25\tau)$       &       \begin{tabular}{c}
                                            $j(\tau)=(j_{25}^{10}+250j_{25}^{9}+4375j_{25}^{8}+35000j_{25}^7\phantom{XXXXXXX}$\\
                                            $ + 178125j_{25}^6+ 631250j_{25}^5+ 1640625j_{25}^4\phantom{XXX}$\\
                                            $ + 3125000j_{25}^3 + 4296875j_{25}^2+ 3906250j_{25} \phantom{XX}  $\\
                                                $\phantom{XX}+ 1953125)^3/j_{25}^{25}(j_{25}^4+5j_{25}^3+15j_{25}^{2}+25j_{25}+25)$
                                         \end{tabular}                  \\ \hline
            \end{tabular}
            \caption{Hauptmodul $j_N$ for $\Gamma_0(N)$ and algebraic equations of $j$-functions.}
            \label{Table1}
\end{center}
\end{table}
Referring to \cref{Table1}, we deduce the main term by the same method used in \cref{se3}, and obtain the asymptotic formula for the Fourier coefficients of the $j$-function again.
Since we must confirm that the strong Gaussian condition is satisfied in each case, we prove it for a more general functions $g(t)$, including the Hauptmoduln discussed in this paper.
\begin{theorem}
For any positive integers $a$, $b$ with $a|b$ and $c$, we define
$$
g(t):=\prod_{k=1}^{\infty}\left(\frac{1-t^{bk}}{1-t^{ak}}\right)^{c}.
$$
Then the function $g(t)$ satisfies the strong Gaussian condition.
\end{theorem}
\begin{proof}
Let
$$
g_{k}(t):=\left(\frac{1-t^{bk}}{1-t^{ak}}\right)^{c}=\sum_{n=0}^{\infty}g_{k, n}t^n.
$$
Note that all $g_{k, n}\ge 0$ because of the condition $a|b$.
Let $X_{t, k}$ be a random variable such that
$$
P[X_{t, k}=n]=\frac{g_{k, n}t^n}{g_{k}(t)}.
$$
Again, we denote the mean of $g_{k}(t)$ by $m_{k}(t)$ and the variance by $\sigma^{2}_{k}(t)$, and define $Y_{t, k}:=X_{t, k}-m_{k}(t)$. Then, we obtain
\begin{align*}
E[Y^{4}_{t, k}]\le C\frac{k^4 t^{ak}}{(1-t^{ak})^4}
\end{align*}
for some $C>0$. Since
\begin{align*}
\sum_{k=1}^{\infty}\frac{k^4 e^{-ak\rho}}{(1-e^{-ak\rho})^4}
=-\frac{4(\zeta(2)+3\zeta(3)+2\zeta(4))}{(a\rho)^5}+\mathcal{O}\left(\frac{1}{\rho^4}\right)
\end{align*}
hold, and
\begin{align*}
\sigma^2_g(t)=c(\sigma^2_{P_{a,a}}(t)-\sigma^2_{P_{b,b}}(t))=\frac{c(b-a)\pi^2}{3ab\rho^3}+\mathcal{O}\left(\frac{1}{\rho^2}\right)
\end{align*}
is obtained from \eqref{oP}, we obtain the evaluation
\begin{align*}
\lim_{t\to 1}\frac{1}{(\sigma^2_g(t))^2}\sum_{k=1}^{\infty}E[Y_{t, k}^4]\le \lim_{t\to 1}\left(\frac{3ab\rho^3}{c(b-a)\pi^2}\right)^2\frac{4C(\zeta(2)+3\zeta(3)+2\zeta(4))}{(a\rho)^5}=0.
\end{align*}
\end{proof}
For example, we discuss the case $N=2$. According to \cref{Table1}, we see that
\begin{align*}
j_2(\tau)
&=\frac{\eta(\tau)^{24}}{\eta(2\tau)^{24}}=\frac{1}{q}\prod_{n=1}^{\infty}\frac{(1-q^n)^{24}}{(1-q^{2n})^{24}}=\frac{1}{q}\prod_{n=1}^{\infty}\frac{1}{(1+q^n)^{24}}\\
&=\frac{1}{q}- 24 + 276 q - 2048 q^2 + 11202 q^3 - 49152 q^4 + 184024 q^5 -\cdots\\
&:=\frac{1}{q}-24+\sum_{n=1}^{\infty}d_n^{(2)}q^n
\end{align*}
and
$$
j(\tau)=j_2(\tau)+3\cdot2^8+3\cdot2^{16}j_2(\tau)^{-1}+2^{24}j_2(\tau)^{-2}
$$
holds. Therefore, we consider the main term $j_2$ in the same way as $H_1$, and the terms $3\cdot2^{16}j_2^{-1}$ and $2^{24}j_2^{-2}$ similarly to $H_2$. By obtaining asymptotic formulas for the coefficients of each term, we identify the main term. 
First, for $j_2$, since 
$$
-\frac{1}{q}\prod_{n=1}^{\infty}\frac{(1-(-q)^n)^{24}}{(1-(-q)^{2n})^{24}}=-\frac{1}{q}\prod_{n=1}^{\infty}(1+q^{2n-1})^{24}\prec0,
$$
it suffices to focus on 
\begin{align*}
j_2^{*}(\tau)&:=\frac{1}{q}\prod_{n=1}^{\infty}(1+q^{2n-1})^{24}\\
&=\frac{1}{q}+ 24 + 276 q + 2048 q^2 + 11202 q^3 + 49152 q^4 + 184024 q^5 +\cdots
\end{align*}
when computing the Fourier coefficients of $j_2$.
From \eqref{mQ}, \eqref{oQ}, and \eqref{lQ}, since the mean and the variance of $j_2^{*}$ are 
$$
-1+\frac{\pi^2}{\rho^2}+\mathcal{O}\left(\frac{1}{\rho}\right)\qquad\text{and}\qquad \frac{2\pi^2}{\rho^3}+\mathcal{O}\left(\frac{1}{\rho^2}\right),
$$
respectively, and because of 
$$
\frac{\pi^2}{\lambda}+\mathcal{O}(\lambda),
$$
the Fourier coefficients of $j_2$ are asymptotically close to
\begin{align}\label{j2asy}
(-1)^{n-1}d_n^{(2)}\sim \frac{e^{2\pi\sqrt{n}}}{2n^{\frac{3}{4}}}.
\end{align}
Now, we need only to compute the Fourier coefficients of $j_2^{-1}$ and $j_2^{-2}$ using \eqref{mP}, \eqref{oP}, and \eqref{lP} given in \cref{sec2}. Since the mean and the variance of $j_2^{-1}$ are given by 
$$
1+\frac{2\pi^2}{\rho^2}+\mathcal{O}\left(\frac{1}{\rho}\right)\qquad\text{and}\qquad\frac{4\pi^2}{\rho^3}+\mathcal{O}\left(\frac{1}{\rho^2}\right),
$$
respectively, and noting 
$$
\log j_2^{-1}(e^{-\lambda})=\frac{2\pi^2}{\lambda}-12\log 2+\mathcal{O}(\lambda),
$$
we obtain that the  Fourier coefficients of $3\cdot2^{16}j_2^{-1}$ are asymptotically close to
\begin{align}\label{j2-1}
48\frac{e^{2\pi\sqrt{2n}}}{(2n)^{\frac{3}{4}}}.
\end{align}
Similarly, in the case of $j_2^{-2}$, the mean and the variance are given by 
$$
1+\frac{4\pi^2}{\rho^2}+\mathcal{O}\left(\frac{1}{\rho}\right)\qquad\text{and}\qquad\frac{8\pi^2}{\rho^3}+\mathcal{O}\left(\frac{1}{\rho^2}\right),
$$
rspectively, and 
$$
\log j_2^{-2}(e^{-\lambda})=\frac{4\pi^2}{\lambda}-24\log 2+\mathcal{O}(\lambda),
$$
holds as well. Therefore, we obtain the Fourier coefficients of $2^{24}j_2^{-2}$ are asymptotically close to 
\begin{align}\label{j2-2}
\frac{e^{4\pi\sqrt{n}}}{\sqrt{2}n^{\frac{3}{4}}}.
\end{align}
From \eqref{j2-1} and \eqref{j2-2}, we see that the Fourier coefficients of $2^{24} j_2^{-2}$ give the main term and we again obtain the asymptotic formula \eqref{jasy} for the Fourier coefficients of the $j$-function.

As in the $N=2$ case, we calculate the other cases listed in \cref{Table1} in a similar way and confirmed that the asymptotic formula for the Fourier coefficients of the $j$-function can also be obtained in all cases. The necessary computational results are summarized in \cref{Table2}. In the case where $N$ is prime, the main term takes the form of a constant multiple of $j_N^{-N}.$
\begin{example}\label{ex1}
Let us look at the case $N=3$. By the formula $t\,d/dt \log F(t)$ in \eqref{kitaichi}, the mean of $j_3^{-3}$ is three times that of $j_3^{-1}$. Therefore, the approximate value of the mean of $j_3^{-3}$ is $4\pi^2/\rho^2$ (see~\cref{Table2}). Similarly, the approximate value of the variance of $j_3^{-3}$ is $8\pi^2/\rho^3$. Moreover, since $\log j_3^{-3}(e^{-\lambda})=3\log j_3^{-1}(e^{-\lambda})$ the main term of $\log j_3^{-3}(e^{-\lambda})$ is given by $4\pi^2/\lambda-18\log 3$. Combining these results, we obtain that the Fourier coefficients of $3^{18}/j_3^{3}$ are  asymptotically close to 
$$
3^{18}\frac{e^{2\pi\sqrt{n}}}{3^{18}\sqrt{2\pi}\left(8\pi^2\left(\dfrac{n}{4\pi^2}\right)^{\frac{3}{2}}\right)^{\frac{1}{2}}e^{-2\pi\sqrt{n}}}=\frac{e^{4\pi\sqrt{n}}}{\sqrt{2}n^{\frac{3}{4}}}.
$$
\end{example}
On the other hand, when $N$ is not prime, a polynomial $\Phi(j_N)$ which is not a monomial appears in the denominator of the main term. However, each of these polynomials can also be expressed using eta-quotients, as shown in \cref{Table2}, and thus the asymptotic formulas for the $j$-function can be derived by the method described in this paper.

\begin{example}
Next consider the case $N=4$. The mean and the variance of $j_4^{-4}$ are approximately $4\pi^2/\rho^2$ and $8\pi^2/\rho^3$, respectively (four times of these of $j_4^{-1}$, which are given in \cref{Table2}).
On the other hand, the mean and the variance are asymptotically smaller than those of $j_4^{-4}$ as can be read off from \cref{Table2}. 
Therefore, the approximation of the mean of $1/j_4^4(j_4+16)$, which is the sum of these of $j_4^{-4}$ and $1/(j_4+16)$, is $4\pi^2/\rho^2$, and  
the approximation of the variance is given by $8\pi^2/\rho^3$.
Moreover, since the main term of $\log j_4^{-4}(e^{-\lambda})+\log \Phi_4(e^{-\lambda})$ is $4(\pi^2/\lambda-8\log 2)-4\log 2$, the asymptotic formula for the Fourier coefficients of $2^{36}/j_4^{4}(j_4+16)$ becomes 
$$
2^{36}\frac{e^{2\pi\sqrt{n}}}{2^{36}\sqrt{2\pi}\left(8\pi^2\left(\dfrac{n}{4\pi^2}\right)^{\frac{3}{2}}\right)^{\frac{1}{2}}e^{-2\pi\sqrt{n}}}=\frac{e^{4\pi\sqrt{n}}}{\sqrt{2}n^{\frac{3}{4}}}.
$$
\end{example}

In this way, by referring to \cref{Table2}, we can see that the asymptotic formula for the Fourier coefficients of the $j$-function can also be obtained using the alternative expressions of the $j$-function via all the Hauptmoduln in \cref{Table1}.

\begin{remark}
\begin{enumerate}
\renewcommand{\labelenumi}{(\theenumi)}
\item The constant multiple in the main term, which can be determined purely algebraically, is not at all a random number, 
because it should precisely cancel out with the constant that arises in the logarithmic computation. This is quite remarkable. 

\item We expect that, by using other Hauptmoduln not listed in \cref{Table1}, the Fourier coefficients of the $j$-function can also be obtained through a similar method.

\item Equation \eqref{j2asy} gives an asymptotic formula for the Fourier coefficients of the Hauptmodul $j_2$, which is shown in~\cite{MO} together with the case of levels $3$ and $5$. 

Furthermore, we can also consider the Fourier coefficients of the Hauptmodul $ j_4$. 
\begin{align*}
j_4(\tau)&:=\frac{\eta(\tau)^8}{\eta(4\tau)^8}=\frac{1}{q} - 8 + 20q - 62q^3 + 216q^5 - 641q^7 +1636q^9 -\cdots\\
&:=\frac{1}{q}-8+\sum_{n=1}^{\infty}d_n^{(4)}q^{2n-1}.
\end{align*}
We can show that the $n$th Fourier coefficients of the function
\begin{align*}
j_4^{*}(\tau)&:=16\frac{\eta(8\tau)^4}{\eta(2\tau)^4}-\frac{\eta(2\tau)^4}{\eta(8\tau)^4}\\
&=-\frac{1}{q} + 20q + 62q^3 + 216q^5 + 641q^7 +1636q^9+\cdots
\end{align*}
is equal to $|d_n^{(4)}|$. By applying the same method as in~\cref{se3}, we find that the main term of the Fourier coefficients of $j_4^{*}$ is $16\eta(8\tau)^4/\eta(2\tau)^4$. Therefore, by calculating the asymptotic formula for the Fourier coefficients of $16\eta(8\tau)^4/\eta(2\tau)^4$ using \eqref{mP}, \eqref{oP}, and \eqref{lP}, we obtain the following.
$$
(-1)^{n+1}d_n^{(4)}\sim\frac{e^{\pi\sqrt{2n}}}{2^{\frac{9}{4}}n^{\frac{3}{4}}}.
$$
For the Hauptmoduln listed in \cite[Table~3]{CN} other than $j_2$ and $j_4$, the author was unable to find a suitable expression by product expansions with positive (or alternating) coefficients, and thus the asymptotic formulas for their Fourier coefficients cannot be obtained so far using the method of this paper.

\end{enumerate}
\end{remark}
\clearpage
\renewcommand{\arraystretch}{1.6}
\begin{table}[H]
\begin{sideways}
            \begin{tabular}{|c|c|c|c|c|} \hline
               $N$ &        Main term    &          The mean of $j_N^{-1}$      &    The variance of $j_N^{-1}$      &      Main term of $\log j_N^{-1}(e^{-\lambda})$      \\ \hline
               $2$ &         $2^{24}/j_2^2$                &      $2\pi^2/\rho^2+\mathcal{O}(1/\rho)$                &        $4\pi^2/\rho^3+\mathcal{O}(1/\rho^2)$         &        $2\pi^2/\lambda-12\log 2$                 \\ \hline
               $3$  &        $3^{18}/j_{3}^3$       &          $4\pi^2/3\rho^2+\mathcal{O}(1/\rho)$         &      $8\pi^2/3\rho^3+\mathcal{O}(1/\rho^2)$    &     $4\pi^2/3\lambda-6\log 3$             \\ \hline
               $4$  &        $2^{36}/j_4^4(j_4+16)$       &      $\pi^2/\rho^2+\mathcal{O}(1/\rho)$             &     $2\pi^2/\rho^3+\mathcal{O}(1/\rho^2)$     &       $\pi^2/\lambda-8\log 2$        \\ \hline
               $5$  &        $5^{15}/j_5^5$      &        $4\pi^2/5\rho^2+\mathcal{O}(1/\rho)$        &  $8\pi^2/5\rho^3+\mathcal{O}(1/\rho^2)$     &         $4\pi^2/5\lambda-3\log 5$ \\ \hline
               $7$  &        $7^{14}/j_7^7$       &        $4\pi^2/7\rho^2+\mathcal{O}(1/\rho)$         &  $8\pi^2/7\rho^3+\mathcal{O}(1/\rho^2)$   &       $4\pi^2/7\lambda-2\log 7$  \\ \hline
               $9$  &        $3^{30}/j_9^9(j_9^2+9j_9+27)$          &       $4\pi^2/9\rho^2+\mathcal{O}(1/\rho)$      &  $8\pi^2/9\rho^3+\mathcal{O}(1/\rho^2)$      &        $4\pi^2/3\lambda^2-3\log 3$   \\ \hline
               $13$  &      $13^{13}/j_{13}^{13}$      &       $4\pi^2/13\rho^2+\mathcal{O}(1/\rho)$         &       $8\pi^2/13\rho^3+\mathcal{O}(1/\rho^2)$        &           $4\pi^2/13\lambda-\log 13$    \\ \hline
               $25$  &       \begin{tabular}{c}
                                      $5^{27}/j_{25}^{25}(j_{25}^4+5j_{25}^3\phantom{XXXX}$  \\      $\phantom{XX}+15j_{25}^{2}+25j_{25}+25)$    \end{tabular}        &   $4\pi^2/25\rho^2+\mathcal{O}(1/\rho)$    &   $8\pi^2/25\rho^3+\mathcal{O}(1/\rho^2)$     &          $4\pi^2/25\lambda-\log 5$        \\ \hline \hline
              $N$ &       \begin{tabular}{c}  Eta quotient of the polynomial $\Phi_{N}(j_N)$\\ appeared in the denominator \end{tabular}    &       
                                            The mean $m$ of $\Phi_N$       &         
                                       The variance $\sigma^2$ of $\Phi_N$      &      Main term of $\log \Phi_{N}(e^{-\lambda})$      \\ \hline
                                        $4$ &         $1/(j_4+16)=\eta(\tau)^8\eta(4\tau)^{16}/\eta(2\tau)^{24}$                &      $\mathcal{O}(1/\rho)$                &        $\mathcal{O}(1/\rho^2)$         &        $-4\log 2$                 \\ \hline
                      $9$  &        $1/(j_9^2+9j_9+27)=\eta(\tau)^{3}\eta(9\tau)^9/\eta(3\tau)^{12}$       &          $\mathcal{O}(1/\rho)$         &      $\mathcal{O}(1/\rho^2)$    &     $-3\log 3$             \\ \hline
                     $25$  &       \begin{tabular}{c}  $1/(j_{25}^4+5j_{25}^3+15j_{25}^{2}+25j_{25}+25)$\\$=\eta(\tau)\eta(25\tau)^5/\eta(5\tau)^{6}$   \end{tabular}    &      $\mathcal{O}(1/\rho)$          &     $\mathcal{O}(1/\rho^2)$     &       $-2\log 5$        \\ \hline

            \end{tabular}   
\end{sideways}
\caption{}
\label{Table2}
\end{table}

\clearpage

\section*{Acknowledgements}
The author is deeply grateful to Professor Ram Murty for introducing me to the new  research topic.
The author would also like to express their sincere gratitude to Professor Masanobu Kaneko for helpful comments regarding modular forms, and to Professor Tomoyuki Shirai for insightful advice on probability theory.
This work is supported by JSPS KAKENHI Grant Number 25KJ1953 and WISE program (MEXT) at Kyushu University.



\begin{thebibliography}{9}

\bibitem{D} L.~B\'aez-Duarte, Hardy--Ramanujan's asymptotic formula for partitions and the central limit theorem, Adv. Math., {\bf 125} (1) (1997), 114--120.

\bibitem{B} P.~Billingsley, Probability and Measure, anniversary edition, John Wiley \& Sons, Hoboken, NJ, (2012). 

\bibitem{Z}J.~H.~Bruinier, G.~van der Geer, G.~Harder and D.~Zagier, The 1-2-3 of Modular Forms, Universitext, Springer-Verlag, Belin--Heidelberg (2008).

\bibitem{CN} J.~H.~Conway and S.~P.~Norton, Monstrous Moonshine, Bull. London Math. Soc. {\bf 11} (1979), 308--339.

\bibitem{DM} M.~Dewar and M.~R.~Murty, An asymptotic formula for the coefficients of $j(z)$, Int. J. Number Theory, {\bf 9} (2013), 641--652. 

\bibitem{F} R.~Fricke, Lehrbuch der Algebra, Dritter Band, Druck und Friedr. Vieweg \& Sohn, (1924). 

\bibitem{IM} K.~Ikeda and R.~Murty, The distinct partition function via probability, to appear in the Proceedings of the Alladi Ramakrishnan Centenary Conference.

\bibitem{K} M.~Kaneko, Traces of singular moduli and the Fourier coefficients of the elliptic modular function $j(\tau)$. CRM Proc. Lect. Notes (Ottawa, ON, 1996) (Number Theory, 19). American Mathematical Society, Providence, RI, (1999), 173--176. 

\bibitem{MO}T.~Matsusaka and R.~Osanai, Arithmetic formulas for the Fourier coefficients of Hauptmoduln of level 2, 3, and 5, Proc. Amer. Math. Soc. 145, (4) (2017), 1383--1392.

\bibitem{MS} M.~R.~Murty and K.~Sampath, On the asymptotic formula for the Fourier coefficients of $j$-function, Kyushu J. Math. {\bf 70} (2016), no.1, 83--91.

\bibitem{P} H.~Petersson, \"{U}ber die Entwicklungskoeffizienten der automorphen Formen, Acta Math. {\bf 58} (1) (1932) 169--215.

\bibitem{R} H.~Rademacher, The Fourier coefficients of the modular invariant $j(\tau)$, Amer. J. Math. {\bf 60} (2) (1938) 501--512.

\end{thebibliography}
\end{document}